\documentclass[12pt]{article}

\usepackage{amsmath,amsthm,amssymb, color}
\usepackage{graphicx, hyperref}
\usepackage{amsthm,amsfonts}
\usepackage[english]{babel}
\usepackage{makeidx}
\usepackage{amsfonts}
\usepackage{amscd}
\usepackage{graphicx}
\usepackage{amssymb}
\usepackage{amstext}

\usepackage{mathrsfs}

\usepackage{subfig}

\newtheorem{theorem}{Theorem}
\newtheorem{proposition}[theorem]{Proposition}
\newtheorem{corollary}[theorem]{Corollary}
\newtheorem{lemma}[theorem]{Lemma}
\newtheorem{definition}{Definition}
\newtheorem{remark}{Remark}
\newtheorem{exam}{Example}

\let\m=\mu

\let\P=\Pi

\newcommand{\R}{\mathbb{R}}

\newcommand{\LLW}{\mathcal{L}_{A}}

\makeindex

\begin{document}

\title{Level-2 IFS Thermodynamic Formalism: Gibbs probabilities in the space of probabilities and the push-forward map}
\author{A. O. Lopes and E. R. Oliveira}
\maketitle

\begin{abstract}
 We will denote by $\mathcal{M}$ the space of Borel probabilities on the  symbolic space $\Omega=\{1,2...,m\}^\mathbb{N}$. $\mathcal{M}$ is equipped Monge-Kantorovich metric. We consider here the push-forward map $\mathfrak{T}:\mathcal{M} \to \mathcal{M}$ as a dynamical system. The space of Borel probabilities on $\mathcal{M}$ is denoted by $\mathfrak{M}$.  Given a continuous function $A: \mathcal{M}\to \mathbb{R}$, an {\it a priori} probability $\Pi_0$ on $\mathcal{M}$, and a certain convolution operation  acting on pairs of probabilities on $\mathcal{M}$, we define an associated Level-2 IFS Ruelle operator.  We show the existence  of an eigenfunction and an eigenprobability $\hat{\Pi}\in\mathfrak{M}$ for such an operator. Under a normalization condition for $A$, we  show the existence of some  $\mathfrak{T}$-invariant probabilities $\hat{\Pi}\in\mathfrak{M}.$ We are able to define the variational entropy of such $\hat{\Pi}$ and a  related maximization pressure problem associated to $A$. In some particular examples, we show how  to get eigenprobabilities solutions  on $\mathfrak{M}$ for the  Level-2   Thermodynamic Formalism problem from  eigenprobabilities on $\mathcal{M}$ for the classical (Level-1) Thermodynamic Formalism.  These examples highlight the fact that our approach is a natural generalization of the classic case.
\end{abstract}

\vspace{2mm}


\vspace{2mm}



\bigskip

\bigskip

\bigskip

\section{Introduction}

Denote by $\mathcal{M}$ the space of Borel probabilities on the  symbolic space $\Omega=\{1,2...,m\}^\mathbb{N}$. We consider here the push-forward map $\mathfrak{T}:\mathcal{M} \to \mathcal{M}$ as a dynamical system (see Definition \ref{kler12}).
First, we will briefly  investigate the dynamical properties  of the push-forward map in  Section \ref{pu} (related results appear in \cite{Bauer}, \cite{BVerm}, \cite{Verm} and\cite{Fagner}). Later, given  a continuous function (a potential) 
$A: \mathcal{M}\to \mathbb{R}$ we will introduce an associated Ruelle  operator
acting on continuous functions $f: \mathcal{M}\to \mathbb{R}$,  and we will present a version of the Ruelle Theorem about the existence of eigenvalues, eigenprobabilities, etc... For the classical Ruelle Theorem see \cite{PP90} (or  \cite{BCLMS23}, \cite{LMMS}, \cite{L1}).  

$\mathcal{M}$ is equipped Monge-Kantorovich metric (see \cite{Villa1} and \cite{Villa2}).
The space of Borel probabilities on $\mathcal{M}$ is denoted by $\mathfrak{M}$. In order to define our Ruelle operator it will be essential to consider an {\it a priori} probability $\Pi_0$ on $ \mathcal{M}$, and the introduction of a certain convolution operation acting on pairs of probabilities on $\mathcal{M}$ (see Section \ref{concon}); it will be also necessary to combine this convolution with the action of the  push-forward map $\mathfrak{T}$ (see Section \ref{IFS1}).

At the beginning of Subsection \ref{oo1} we present the main assumptions for defining an IFS Ruelle operator $B_{\Pi_0}$ on our setting, in order to  be able to obtain (after some work), from already known general results on IFS, the main conclusions of the paper.   For example, one of our main results is  Theorem \ref{exist posit eigenfunction} which claims

\begin{theorem}\label{exist posit eigenfunction22}
 If $A: \mathcal{M}  \to \mathbb{R}$ is a Lipschitz potential, then there exists a positive and continuous eigenfunction $h: \mathcal{M}  \to \mathbb{R}$, such that, $B_{\Pi_0}(h) = \lambda h$, $\lambda>0.$
\end{theorem}

We will provide examples later on  in the text (see Examples \ref{prok}, \ref{prof} and \ref{prok1});  they will clarify to the reader the unequivocal fact that the results obtained in our setting are a natural generalization of classical Thermodynamical Formalism (in the sense of \cite{PP90}); which can be considered the Level-1 setting.

It will be natural to consider in our Level-2 setting the concept of variational entropy of a holonomic probability, the pressure problem, and equilibrium probabilities (see Definitions  \ref{entropy} and \ref{lkle}  on Subsection \ref{oo1}). Later, we present our main result which  is the relation between the Ruelle Theorem and the equilibrium probability (see expression \eqref{kler39}).
In the Example \ref{prof1} we show that our formalism can be used to provide examples of  $\mathfrak{T}$-invariant probabilities on $\mathcal{M}.$

General references in Thermodynamic Formalism for  IFS are \cite{Arbieto2016WeaklyHyperbolic}, \cite{BOR23}, \cite{cioletti2019}, \cite{Urb},  \cite{LOPV},   \cite{LM1}, \cite{LO} and \cite{Miha}.
 
\section{The push-forward map  acting in the space of probabilities on the symbolic space} \label{pu}

In the present section, we will describe preliminary results (we also present several examples to facilitate the understanding of the theory) that will be needed later in other sections.


We consider the shift acting on the symbolic space $\Omega=\{1,2,...,m\}^\mathbb{N}.$  In $\Omega$ we consider the usual metric $d=d_{\Omega}:\Omega^2 \to \mathbb{R} $ which makes $\Omega$ a compact space:
\begin{equation}\label{dist omeg} d_{\Omega}(\alpha, \beta):=\left\{
  \begin{array}{ll}
    0, & \alpha=\beta \\
    \frac{1}{2^{k}}, & k=\min \alpha_i \neq \beta_i
  \end{array}
\right.
\end{equation}
for any $\alpha, \beta \in \Omega$.

As we mentioned before, we denote by $\mathcal{M}$ the set of probabilities on the Borel sigma-algebra  which is a compact convex space when considering the Hutchinson distance (also called  Monge-Kantorovich or 1- Wasserstein) $d_{MK}: \mathcal{M}^2 \to \mathbb{R}$
defined by
\begin{equation}\label{Wass1}
d_{MK}(\mu,\nu)=\sup_{f \in {\rm Lip}_{1}(\Omega)} \int_{\Omega}f d\mu - \int_{\Omega}f d\nu,
\end{equation}
for any  $\mu, \nu \in \mathcal{M}$, which equivalent to the weak-$\ast$ convergence because $\Omega$ is compact, see \cite{Arbieto2016WeaklyHyperbolic} Theorem 1.6.

Note that if
$ d(x_0,x_1)\leq \epsilon,$ then $\,d_{MK}(\delta_{x_0},\delta_{x_1})\leq \epsilon.$

 We denote by $\mathcal{C}= C(\Omega,\mathbb{R})$ the set of real continuous functions with domain $\Omega$ and by $\mathfrak{C}=C(\mathcal{M},\mathbb{R})$ the set of real continuous functions (using the Monge-Kantorovich metric $d_{MK}$) with domain $\mathcal{M}.$

We denote $\mathcal{M}_\sigma^i$ the set of $\sigma$-invariant probabilities and by $\mathcal{M}^e_\sigma$ the set of $\sigma$-ergodic probabilities.

\begin{definition} \label{kler12} Given   probability $\mu_1\in \mathcal{M}$, the push-forward of $\mu_1$ is the probability $\mathfrak{T}(\mu_1)=\mu_2$ such that for all
Borel set $E$ we get that  $\mu_2 (E) = \mu_1 (\sigma^{-1} (E)).$ $\mathfrak{T}$ is called the {\bf push-forward map} acting on the space of probabilities on $\Omega$.
\end{definition}

To say that $\mu$ is $\sigma$-invariant is the same that to say that $\mathfrak{T}(\mu)=\mu.$

Equivalently, for any $f \in C(\Omega, \mathbb{R})$
\begin{equation} \label{pum0} \int f \,d \,\mathfrak{T} (\mu_1) = \int (f \circ \sigma)\, d \mu_1.
\end{equation}

In particular,  
\begin{equation} \label{pum000} \mathfrak{T}(\delta_{x})=\delta_{\sigma (x)}.\end{equation}

Note that if $x_1,x_2$ are such that  $\sigma(x_1)=x_0= \sigma(x_2)$, then, 
\begin{equation} \label{pum1}\mathfrak{T}(\delta_{x_1})=\delta_{\sigma (x_1)}=\delta_{x_0} =\mathfrak{T}(\delta_{\sigma(x_1)})=\mathfrak{T}(\delta_{x_2}).\end{equation}

Moreover,
\begin{equation} \label{pum2} \mathfrak{T}^n(\delta_{x_1})= \delta_{\sigma^n (x_1)}
.\end{equation}

We denote by $\mathfrak{M} $ the set of probabilities on the Borel sigma-algebra of $\mathcal{M}$ which is a non-empty
compact separable convex space, when considering a metric $d_{MK}$ associated to the weak-$\ast$ convergence (the
 Monge-Kantorovich metric for instance). We denote $\mathfrak{M}_\mathfrak{T}$ the set of $\mathfrak{T}$-invariant probabilities and by $\mathfrak{M}^e_\mathfrak{T}$ the set of $\mathfrak{T}$-ergodic probabilities.

It is important not to confuse the concept that a probability measure  $\mu \in \mathcal{M}$ is invariant for $\mathfrak{T}$, in the sense of $\mathfrak{T}(\mu)= \mu$, with the statement that  a probability measure $\Pi\in \mathfrak{M}$   is invariant for the dynamical  transformation $\mathfrak{T} :\mathcal{M} \to \mathcal{M}$, that is 
$\Pi \in \mathfrak{M}_\mathfrak{T}^i.$ The later means: for any continuous function $F:\mathcal{M} \to \mathbb{R}$
\begin{equation} \label{puru2} \int F(\rho) d \Pi (\rho) = \int (F \circ \mathfrak{T})(\rho) d \Pi (\rho).
\end{equation}

Via the Ruelle operator, we will show the existence of nontrivial $\mathfrak{T}$-invariant probabilities in Example \ref{prof1} (see also Remark \ref{sim1}).

\begin{remark} \label{pper} As $\mathfrak{T}^n(\delta_{x_1})= \delta_{\sigma^n (x_1)}$, we get that in the  case
$\sigma^n(x_1)= x_1$, then, $\delta_{x_1}, \delta_{\sigma(x_1)},..., \delta_{\sigma^{n-1}(x_1)}$ is a periodic orbit for
of period $n$ for $\mathfrak{T}$. Note also that
$\mathfrak{T} (\sum_{j=1}^ k p_j \, \delta_{x_j})= \sum_{j=1}^ k p_j \, \delta_{\sigma(x_j)},$
where $\sum_{j=1}^ k p_j =1$, $p_j \geq 0.$

Then,
$\sum_{j=1}^ k p_j \, \delta_{x_j}\in \mathfrak{T}^{-1} ( \sum_{j=1}^ k p_j \, \delta_{\sigma(x_j)}).$

If $\mu$ is $\sigma$-invariant, as $  \mathfrak{T}(\mu) = \mu$, we get that  $\mu \in \mathfrak{T}^{-1}(\mu) .$

Therefore, if $\sigma^n(x_1)= x_1$, then
\begin{equation} \label{pum3} \mathfrak{T} ( \frac{1}{n} \sum_{j=0}^ {n-1} \, \delta_{\sigma^j (x_1)})= ( \frac{1}{n} \sum_{j=0}^ {n-1} \, \delta_{\sigma^j (x_1)})
\end{equation}
and
$ \frac{1}{n} \sum_{j=0}^ {n-1} \, \delta_{\sigma^j (x_1)}\in \mathfrak{T}^{-1} ( \frac{1}{n} \sum_{j=0}^ {n-1} \, \delta_{\sigma^j (x_1)}). $

\end{remark}

The transformation $\mathfrak{T}:\mathcal{M} \to \mathcal{M}$ is continuous (see \cite{Bauer}), takes probabilities to probabilities  and is not injective.

\begin{exam} \label{ex1}

Suppose $\sigma (\tilde{x}) = \tilde{x}.$ Then, $\delta_{ \delta_{\tilde{x}}}
$ is $\mathfrak{T}$-invariant.


More generally, if $\mu$ is $\sigma$-invariant, then $\Pi=\delta_{\mu}\in \mathfrak{M}$ is $\mathfrak{T}$-invariant. Indeed, given a continuous function $f: \mathcal{M} \to \mathbb{R}$, we get that
\begin{equation} \label{pum4}\int (f \circ \mathfrak{T})  \, d\, \delta_{\mu}=  f ( \,\mathfrak{T}\,( \mu )\,) = f (  \mu )= \int  f \, d\,\delta_{ \mu}.
\end{equation}

Then, $\delta_{ \mu}  \in \mathfrak{M}_\mathfrak{T}^i$. That is $\mathfrak{T}^*$ acting on $\mathfrak{M}$ is such that $\Pi =\delta_{ \mu}  $ satisfies  $\mathfrak{T}^* (\Pi)=\P$, that is $\mathfrak{T}^* (\delta_{\mu} )=\delta_{\mu}$.

More generally given $\mu_j\in \mathcal{M}_\sigma^i$, $j=1,2,..,k$, then, when $\sum_{j=1}^ k p_j=1$, $p_j\geq 0$, $j=1,2,..,k$
\begin{equation} \label{pum5} \mathfrak{T}^* (\sum_{j=1}^ k p_j \, \delta_{\mu_j})= \sum_{j=1}^ k p_j \, \delta_{\mu_j}.
\end{equation}

Therefore, $\sum_{j=1}^ k p_j \, \delta_{\mu_j}\in \mathfrak{M}_\mathfrak{T}$

\end{exam}

Note that
\begin{equation} \label{WD} d_{MK} (\delta_{x_0} , \delta_{y_0})\leq d(x_0,y_0).
\end{equation}

Moreover, if $\mu_n \to \mu$, then, $\mathfrak{T} (\mu_n) \to \mathfrak{T} (\mu)$.

Note that $\mathfrak{T}$ is not a $d$ to $1$ map:
consider $x \neq y$, in $\Omega$ such that $\sigma(x)=\sigma(y)=z$, and the family $\mu_{t}=t \delta_x + (1-t) \delta_y$ for $t \in [0,1]$, then
$$\mathfrak{T}(\mu_{t}) = t \delta_{\sigma(x)} + (1-t) \delta_{\sigma(x)} =\delta_z, \, \forall t$$
thus, $\mathfrak{T}^{-1}$ contains infinitely many distinct measures (Lemma \ref{ouyt} also confirms this claim).

\smallskip
For $x=(x_1,x_2,..,x_n,...)$ and a symbol $a$ denote $a \,x =(a,x_1,x_2,..,x_n,...).$
Given a Holder potential $A:\Omega \to \mathbb{R}$,  the Ruelle operator $\LLW$ acts on continuous functions $\psi:\Omega \to \mathbb{R}$ via: 
\begin{equation} \label{belo214}\LLW(\psi)(x)= \sum_{a=1}^m e^{A(ax)}\psi(ax), \, \text{for all}\, x\in \Omega.
\end{equation}

The dual of the Ruelle operator $\LLW$, denoted $\LLW^*$,  acts on finite measures on $\Omega$, and to say that 
$\LLW^*(\mu_1)=\mu_2$, means that for any continuous function $\psi$ we have
$$ \int \psi d \mu_2 = \int \LLW (\psi) \, d \mu_1.$$

We say that $\mu_A$ is the eigenprobability for the dual of the Ruelle operator if there exists $\lambda>0$ such that $\LLW^*(\mu_A)= \lambda\, \mu_A$. When $A$ is continuous an eigenprobability always exists, but may exist more than one (however the eigenvalue is unique).  In the case $A$ is Holder it is unique; for all this see \cite{PP90} or \cite{L1}.

We say that $A$ is normalized if $\LLW(1)=1$. In this case it  is usual to write $A$ in the form $A=\log J$, where $J:\Omega \to (0,1)$ is such that for all $x\in \Omega$ we get that  $\sum_{a=1}^m J(a\, x)=1.$ We  call Jacobian such function call $J$.

We say that $\mu$ is a Holder Gibbs probability, if there exists a normalized potential $A=\log J$, such  that, 
$\mathcal{L}_{\log J}^*(\mu)= \LLW^*(\mu)=\mu$. We say that $J$ is the Jacobian  of  the Holder Gibbs probability $\mu$.

The shift transformation $\sigma:\Omega \to \Omega$ is such that $\sigma(x_1,x_2,..,x_n,...)= (x_2,x_3,..,x_n,...).$

Note that for any $x\in \Omega$ we get  \begin{equation}  \label{FuTF} \LLW^*(\delta_{x}) = \sum_{\sigma(y)=x} J(y) \delta_y.
\end{equation}

We denote by $\mathcal{G}$ the set of all  Holder Gibbs probabilities.

\begin{theorem} \label{klh} (see \cite{PP90}) Given a Holder Gibbs probability $\mu$ associated to the Jacobian $J$, and any point $x_0\in \Omega$, we get that in the $1$-Wassertein distance 
$$ \lim_{n \to \infty}  (\mathcal{L}_{\log J}^*)^n\,(\delta_{x_0}) = \mu.$$

\end{theorem}

 In the case, $A$ is Holder \cite{KLS} shows that the convergence is exponential  in the $1$-Wassertein distance. The support of the probability $(\mathcal{L}_{\log J}^*)^n\,(\delta_{x_0})$ is in the set of $n$-preimages of $x_0$ by $\sigma$.

\begin{theorem} \label{klh1}  (see \cite{L3}) The set $\mathcal{G}$ is dense in the set
 $\mathcal{M}_\sigma^i$.

\end{theorem}

\begin{theorem} \label{klh2}  (see \cite{LM}, \cite{PP90} and also \cite{Parry}) Given a probability $\mu$ in 
$\mathcal{G}$, it can be weakly  approximated by a probability $\rho$, which is  a finite convex combination of probabilities with  support in periodic orbits. Of course, $\rho$ is a periodic orbit for $\mathfrak{T}.$

\end{theorem}

\begin{lemma} \label{ouyt} The transformation $\mathfrak{T}:\mathcal{M} \to \mathcal{M}$ is surjective over $\mathcal{M}$. This follows from the fact that when $A$ is normalized, 
\begin{equation} \label{ELfun}\text{if}\,\, \mathcal{L}_A^* (\nu) =  \mu,\,\, \text{then}\,\,\mathfrak{T} (\mu)=\nu.
\end{equation}

\end{lemma}

\begin{proof} Given $\nu\in \mathcal{M} $, is there exist $\mu \in \mathcal{M}$ such that $\mathfrak{T}(\mu)=\nu$?

Suppose that $A$ is {\bf any} Holder normalized potential, then, take $\mu= \mathcal{L}_A^* (\nu).$

For any continuous $f$ we get that
$$ \int f d \mathfrak{T}( \mu) =\int (f \circ \sigma) d \mu= \int (f \circ \sigma) d \mathcal{L}_A^* (\nu)= \int  \mathcal{L}_A (f  \circ \sigma)  d \nu=$$
$$ \int f \mathcal{L}_A(1) d \mu = \int f d  \nu.$$

Therefore, 
\begin{equation}\label{gter} \mathfrak{T}(\mu)= \nu.
\end{equation}

In \cite{LR3} it is shown that if $\mu_1$ is Holder equilibrium, then $\mathcal{L}_A^* (\mu_1)$ is not $\sigma$-invariant (unless it is the unique fixed point). Therefore, given a Holder Gibbs probability $\nu$, there exists
preimages $\mu$ of $\nu$ by $ \mathfrak{T}$, such that, are not $\sigma$-invariant.
\end{proof}

\begin{theorem}  \label{Bai} Given $\epsilon>0$,  a probability $\tilde{\mu}_2\in \mathcal{M}$, and $\sigma$-invariant probability $\tilde{\mu}_1$,
 there exist probabilities $\rho_1$ and $\mu_2$ in $\Omega$,  and $N>0$, such that
$$ d_{MK} (\rho_1, \tilde{\mu}_1)< \epsilon,\, d_{MK} (\mu_2, \tilde{\mu}_2)< \epsilon,\,\, \text{and}\,\,   \mathfrak{T}^N(\rho_1) = \mu_2.$$
\end{theorem}

\begin{proof} Given the probability $\tilde{\mu}_2$ we get an $\epsilon$-approximation $\mu_2$ of
$\tilde{\mu}_2$ of the form
$$ \mu_2= \sum_{j=1}^k p_j \, \delta_{x_j},$$
where $\sum_{j=1}^k p_j=1.$

From Theorem \ref{klh1} we can $\epsilon/2$-approximate $\tilde{\m}_1$ by a Holder Gibbs probability $\mu_1$ associated to  the Holder Jacobian $J_1$.

From Theorem \ref{klh}, for each $j=1,2,...,k$, we get that for large $N_j$,  the probability $(\mathcal{L}_{\log J_1}^*)^{N_j} (\delta_{x_j})$ is an $\epsilon/2$-approximation of $\mu_1$.  Therefore, for some uniform large $N$ we get that
$$ \rho_1 =\sum_{j=1}^k p_j (\mathcal{L}_{\log J_1}^*)^{N} (\delta_{x_j})= (\mathcal{L}_{\log J_1}^*)^{N}  (\sum_{j=1}^k p_j \,\delta_{x_j})=  (\mathcal{L}_{\log J_1}^*)^{N} (\mu_2)$$
is an  $\epsilon/2$-approximation of $\mu_1$, and therefore  an $\epsilon$-approximation of  $\tilde{\mu}_1$.

It follows from \eqref{gter} in Lemma \ref{ouyt} that
$ \mathfrak{T}^N  (  \rho_1) = \mu_2 .$
\end{proof}

\begin{corollary} There exists a dense orbit for $\mathfrak{T}$ in $\mathcal{M}.$

\end{corollary}

\begin{proof} As there exists a countable dense set of probabilities $\rho_n$, $n \in \mathbb{N},$ in $\mathcal{M}$, the result follows from last result and Baire Theorem. Indeed, for each $k,r \in \mathbb{N}$, take the ball $B(\rho_k, \frac{1}{r})$. From  Baire Theorem and Theorem \ref{Bai} we get that 

$$ \cap_{r,k=1}^\infty \, \cup_{n=1}^\infty \,\mathfrak{T}^{-n}  \left(B(\rho_k, \frac{1}{r})\right)$$
is not empty.
\end{proof}

\begin{exam}\label{ex2} Consider a probability $\mu \in \mathcal{M}$ and a natural number $k$. Take the partition $\{ \overline{x_1,x_2,...,x_k}\, , x_r \in \{1,2,...,m\}, r \in \{1,2,...,k\} \}$. Consider the lexicographic order on the set of finite words $(x_1,x_2,...,x_k)$. Now we re-index these words using this order and $\overline{\alpha_j}$ denotes the cylinder associated with the j-th 
word $\alpha_j=\alpha_j^k$, $j=1,2,...,m^k$. Finally, denote by $z_k\in \Omega $ the periodic orbit obtained by the repetition of the string $(\alpha_1,\alpha_2,..., \alpha_{m^k})$. 

Note that for $j>1$,
$$\sigma^{k\, j} ( \alpha_j, \alpha_{j+1},...,\alpha_{m^k},\alpha_1,...,\alpha_{j-1}, \alpha_j,... ) =$$
$$(\alpha_{j+1},...,\alpha_{m^k},\alpha_1,...,\alpha_{j-1}, \alpha_j,\alpha_{j+1},... ) .$$

Therefore, there exists a value $r_k= r^k \, k $, such that, $\sigma^{r_k} (z_k)=z_k.$

From \eqref{pum2}
$$ \mathfrak{T}^{r_k} ( \sum_{j=1}^{m^k} \mu(\overline{\alpha_j}) \delta_{ (\alpha_j, \alpha_{j+1},...,\alpha_{m^k},\alpha_1,...,\alpha_{j-1}, \alpha_j,... ) })=$$
\begin{equation} \label{pum6}
 \sum_{j} \mu(\overline{\alpha_j}) \delta_{ (\alpha_j, \alpha_{j+1},...,\alpha_{m^k},\alpha_1,...,\alpha_{j-1}, \alpha_j,... ) }.
\end{equation}

We denote $\mu_k \in \mathcal{M}$, $k\in \mathbb{N}$, the probability 
\begin{equation} \label{pum7} \mu_k =\sum_{j} \mu(\overline{\alpha_j}) \delta_{ (\alpha_j, \alpha_{j+1},...,\alpha_{m^k},\alpha_1,...,\alpha_{j-1}, \alpha_j,... ) },
\end{equation}
which is periodic of period $r_k$ for  $\mathfrak{T}$. Therefore, $\mu_k \in \mathfrak{M}_{\mathfrak{T}^{r_k}}^i.$

Note that $\mu_k(\overline{\alpha_j}) = \mu(\overline{\alpha_j}) $, and $\mu_k$ is a probability with weights in $\mathfrak{T}$-periodic orbits,  for any $k$. 

\end{exam}

\begin{lemma} \label{ouyt25}

The periodic points of $\mathfrak{T}$ are dense in $\mathcal{M}$.

\end{lemma}

\begin{proof}

Indeed, given any  measure $\mu$ and $\epsilon>0$, take $k$ such that $2^{-k} <\epsilon.$
 The diameter of each cylinder set  $\overline{x_1,x_2,...,x_k}$ is $2^{-k}.$

Consider a Lipchitz function $f$ with Lipschitz constant smaller or equal to $1$; then,  for $s_1,s_2\in \overline{x_1,x_2,...,x_k}$ we get that $|f(s_1)-f(s_2)|\leq 2^{-k}.$

Consider the $\mathfrak{T}$-periodic probability $\mu_k$ of expression \eqref{pum7}. We will show that $d_{MK}(\mu,\mu_k)\leq \epsilon$.

Indeed, 
$$\int f d\mu - \int  f d\nu_k\leq \sum_{j=1}^{m^k} |\int_{\alpha_j} f d \mu - \int_{\alpha_j} f d \mu_k|=$$
$$ \sum_{j=1}^{m^k} |\,\int_{\alpha_j}  f \, d \mu\, - \,f   (\alpha_j, \alpha_{j+1},...,\alpha_{m^k},\alpha_1,...,\alpha_{j-1}, \alpha_j,... )  \, \mu(\alpha_j)\,|\leq $$
\begin{equation}\label{pum9}
 \sum_{j=1}^{m^k}\mu(\alpha_j)
 \, 2^{-k}= 2^{-k} \leq \epsilon.
\end{equation}

\end{proof}


One way to generate probabilities  $\Xi \in\mathfrak{M}$ is the following: take a probability $\nu$
on $\Omega$ and define for each continuous function $F:\mathcal{M} \to \mathbb{R}$
the bounded  linear transformation
\begin{equation} \label{pom1} F\to \Lambda (F) =\int_{\Omega} F( \delta_x ) d \nu(x).
\end{equation}

By Riesz Theorem there exist a probability  $\Xi_\nu $ on $\mathcal{M}$ such that
for all $F \in \mathfrak{C}$ we get
\begin{equation} \label{pom2}  \Lambda(F) = \int_{\mathcal{M}} F(\mu)\, d \,\Xi_\nu\, (\mu).
\end{equation}

We say that $\Xi_\nu\in\mathfrak{M}$ is the {\bf  Level-2 version} of $\nu\in \mathcal{M}.$




\begin{exam} An interesting case is when the $\nu$ above is the maximal entropy $\mu_0$. Given a point $y_0 \in \Omega$ and $n\in \mathbb{N}$, denote by $x_j^m$, $j=1,2,...,m^n$, the $m^n$ solutions of $\sigma^n(x)=y_0.$ Then
\begin{equation} \label{pom12} F\to \Lambda (F) =\int_{\Omega} F( \delta_x ) d \nu(x)=\int_{\Omega} F( \delta_x ) d \mu_0 (x)= \lim_{n\to \infty} \frac{1}{m^n} \sum_{j=1}^{m^n} F(\delta_{x_j^m}).
\end{equation}

Then, in some sense $\Xi_{\mu_0}$ is a Level-2 version of the maximal entropy measure.

\end{exam}


\begin{definition} \label{barit}

Given a probability $\Pi \in \mathfrak{M}$, we call $m_\Pi$ the probability such that $\forall f\in C(\Omega)$
$$ \int_\mathcal{M} \nu(f)\,d\, \Pi\, (\nu)=m_\Pi (f)$$
the barycenter of $\Pi$.
\end{definition}

It is natural to say that $m_\Pi\in \mathcal{M}$ is the {\bf Level-1 version} of  $\Pi \in \mathfrak{M}$.

In this way: for any continuous function $f: \Omega \to \mathbb{R}$

\begin{equation} \label{eees}\int_\Omega f(x) d m_\Pi(x) = \int \,\left(\int_\Omega f(y)\, d\rho(y) \,\right)\, d \Pi\,(\rho).
\end{equation}

It is is easy to see that $m_{\delta_\rho}= \rho$ for any $\rho \in \mathcal{M}.$

The map $\Pi \to m_\Pi$ is continuous when using Monge-Kantorovich metric $d_{MK}$ obtained from $d$. The map $\Pi \to m_\Pi$ is a weak contraction  (see  Proposition \ref{tri}).

\begin{proposition}  \label{tri} The map $\Pi \to m_\Pi$ is a weak contraction.
\end{proposition}
\begin{proof}Indeed,
$d_{MK}(m_{\Pi_{1}},m_{\Pi_{2}}) = \sup\left\{\, \int_{\Omega} f d m_{\Pi_{1}}- \int_{\Omega} f d m_{\Pi_{2}}\, \,| \, \text{Lip}\, {f} \leq 1\,\right\}$, thus we need to evaluate
$$\int_{\Omega} f d m_{\Pi_{1}}- \int_{\Omega} f d m_{\Pi_{2}} = \int_{\mathcal{M}} \nu(f) d \Pi_{1}(\nu)- \int_{\mathcal{M}} \nu(f) d \Pi_{2}(\nu).$$
Define $G(\nu)= \nu(f)$. We claim that ${\rm Lip}(G) \leq 1$. Indeed,
$$G(\nu) - G(\nu')= \nu(f) -\nu'(f) \leq d_{MK}(\nu, \nu'),$$
because $\text{Lip}\, {f} \leq 1$. By definition,
$$\int_{\mathcal{M}} \nu(f) d \Pi_{1}(\nu)- \int_{\mathcal{M}} \nu(f) d \Pi_{2}(\nu) =$$ $$= \int_{\mathcal{M}} G(\nu) d \Pi_{1}(\nu)- \int_{\mathcal{M}} G(\nu) d \Pi_{2}(\nu) \leq d_{MK}(\Pi_{1}, \Pi_{2}),$$
because ${\rm Lip}(G) \leq 1$. Thus, $d_{MK}(m_{\Pi_{1}},m_{\Pi_{2}}) \leq d_{MK}(\Pi_{1}, \Pi_{2}).$\\

It is not a contraction, indeed, take $\Pi_{i}=\delta_{\delta_{x_i}}, \; i=1,2$ then
$$\int f(x) d m_{\Pi_{i}}(x) = \int \int f(x)\, d\nu(x) \,d \delta_{\delta_{x_i}}\,(\nu) =f(x_i),$$
so that $m_{\Pi_{i}}=\delta_{x_i}$. We recall that $G(\nu)= \nu(f)$ satisfy ${\rm Lip}(G) \leq 1$ provided that ${\rm Lip}(f) \leq 1$ (w.r.t. the respective metrics). Thus,
$$d_{MK}(\Pi_{1}, \Pi_{2}) \,\geq \sup_{G(\nu)= \nu(f), \; {\rm Lip}(f) \leq 1 } \int_{\mathcal{M}} G(\nu) d \Pi_{1}(\nu)- \int_{\mathcal{M}} G(\nu) d \Pi_{2}(\nu)=$$ $$= \sup_{{\rm Lip}(f) \leq 1 } \delta_{x_1}(f)  -  \delta_{x_2}(f)=d_{MK}(\delta_{x_2}, \delta_{x_2})= d_{MK}(m_{\Pi_{1}},m_{\Pi_{2}}).$$

From the other inequality we  get that $d_{MK}(m_{\Pi_{1}},m_{\Pi_{2}}) = d_{MK}(\Pi_{1}, \Pi_{2}),$ so the  weak contraction  is not  a contraction.
\end{proof}


\smallskip


\bigskip

Each $\mu \in \mathcal{M}^i_\sigma$ can be associated to a probability $\Theta_{\mu}$ on $\mathcal{M}^e_\sigma$ such that $m_{ \Theta_{\mu}}= \mu$  (see Theorem  6.4 in \cite{Mane} or next proposition). In this case, for any continuous $f:\Omega \to \mathbb{R}$ we get
\begin{equation} \label{poxa}  \int f(x)   d \mu(x) = \int \,\left(\int f(y) d \nu (y) \,\right)d \Theta_{\mu} (\nu).
\end{equation} The support of $ \Theta_{\mu}$ is the set of $\sigma$-ergodic probabilities.

$\Theta_{\mu}$ is called the ergodic decomposition of the $\sigma$-invariant probability $\mu$. Therefore,  $\m$ is the barycenter of $\Theta_{\mu}$.

\begin{proposition} For any $\sigma$-invariant $\mu$ we get that  
$m_{\Theta_\mu} = \mu$

\end{proposition}

\begin{proof} We will show that for any continuous $f$ we get that 
$$\int f(x) d m_{\Theta_\mu} (x)= \int f(x) d \mu(x).$$

From \eqref{poxa} we get
$$\int f (x) d \mu(x)=     \int \,(\int f(y) d \nu (y) \,)d \Theta_{\mu} (\nu),$$
and from \eqref{eees} we get
$$\int f(x) d m_{\Theta_\mu}(x) = \int \,(\int f(x)\, d\nu(x) \,)\,d \Theta_{\mu}\,(\nu).$$

\end{proof}


\bigskip

One can  consider the Level-2 version of the above.

\begin{theorem} \label{mman} (see Theorem  6.1 in \cite{Mane}) For any $\Pi\in \mathfrak{M}_{\mathfrak{T}}$ and any continuous function $\psi: \mathcal{M} \to \mathbb{R}$
$$ \int_\mathcal{M} \psi (\beta) d \Pi(\beta) =  \int_\mathfrak{K} \,  \left(\,\int_{  \mathcal{M}}\,\psi(\gamma) \, d \tilde{\Pi}(\gamma) \,\,\right)\,d\, \mathfrak{O}_{\Pi} (\tilde{\Pi}),     $$
where $\tilde{\Pi} \in\mathfrak{K}$, and $\mathfrak{K}\subset \mathfrak{M}_{\mathfrak{T}}^e$. For each $\Pi \in \mathfrak{M}_{\mathfrak{T}}$ the probability  $\mathfrak{O}_{\Pi}$ on $\mathfrak{M}_{\mathfrak{T}}$ is called the $\mathfrak{T}$-ergodic decomposition of $\Pi$.

\end{theorem}


\smallskip

 In this case $\Pi$ is the barycenter of $ \mathfrak{O}_{\Pi}$.
We will need  a non-dynamical version of the above kind of results.

\smallskip

\begin{remark} \label{kui}
The set of extreme points of the set $\mathcal{M}=\{$ probabilities on $\Omega\}$, is the set (see \cite{H}) 
$$\mathfrak{R}=\{\text{probabilities of the form}\, \, \delta_y\,\text{ where}\, y \, \text{ is any point in }\,\Omega\}\subset \mathcal{M}.$$  

Given $\mu \in \mathcal{M}$, for  some  $\tilde{\Theta}_{\mu}\in \mathfrak{M}$,  we get
\begin{equation} \label{poxat}  \int_\Omega f(x)   d \mu(x) = \int \,\left(\int f(z) d \delta_y (z) \,\right)d \tilde{\Theta}_{\mu} (\delta_y).
\end{equation} The support of $\tilde{\Theta}_{\mu}\in \mathfrak{M}$ is the set $\mathfrak{R}$.

The set of extreme points of the set $\mathfrak{M}=\{$ probabilities on $\mathcal{M}\}$, is the set 
$$\tilde{\mathfrak{K}}=\{\text{probabilities of the form}\,\, \delta_\mu\,\text{ where}\, \mu \, \text{ is any probability in }\,\mathcal{M}\}\subset \mathfrak{M}.$$  

For any $\Pi\in \mathfrak{M}$ there exists $ \mathfrak{O}_{\Pi}$ such that for  any continuous function $\psi: \mathcal{M} \to \mathbb{R}$

\begin{equation} \label{pocat}  \int_\mathcal{M} \psi (\beta) d \Pi(\beta) =  \int_\mathfrak{K} \,  [\,\int_{  \mathcal{M}}\,\psi(\gamma) \, d \tilde{\Pi}(\gamma) \,\,]\, d\,\mathfrak{O}_{\Pi} (\tilde{\Pi}),     \end{equation}
where $\tilde{\mathfrak{K}}$ has probability $1$ for    $\mathfrak{O}_{\Pi}$.

Then, we can write

\begin{equation} \label{abu1} \int_\mathcal{M} \psi (\beta) d \Pi(\beta) =  \int_\mathfrak{M} \,  [\,\int_{  \mathcal{M}}\,\psi(\gamma) \, d  \delta_\mu (\gamma) \,\,]\, \mathfrak{O}_{\Pi} (\delta_{\mu})=\int_\mathfrak{M} \,  \,\psi(\mu)\, \mathfrak{O}_{\Pi} (\delta_{\mu}) .    
\end{equation}

\end{remark}

In this case $\Pi$ is the barycenter of $\mathfrak{O}_{\Pi}$.

\begin{exam} \label{baba3} Given a probability $\mu\in \mathcal{M}$ we can associate, via barycenter, a probability $\hat{\mu}=\Pi_\mu\in \mathfrak{M}$ in the following way: denote $\mathfrak{K}= \{\delta_y, y\in \Omega\}\subset \mathcal{M}$, and then we  associate $\delta_y$ in $\mathcal{M}$ with $y\in \Omega$, and $\delta_{\delta_y}\in \mathfrak{M}$ with $y$. Given a set $\hat{B}\subset \mathfrak{K}$ in $\mathcal{M}$ we associate it to a set $B\in \Omega$ via this association.

Now we  denote by $\hat{\mu}\in \mathfrak{M}$ a probability, where $\mathfrak{K}= \{\delta_y, y\in \Omega\}$ has probability $1$, and such that, given a Borel set  $\hat{C}$ in $\mathcal{M}$
$$ \hat{\mu} (\hat{C}) = \int_\Omega I_{\hat{C}\cap \mathfrak{K}} (\delta_y) d \hat{\mu}(\delta_y) = \mu(\hat{C}\cap \mathfrak{K}).  $$

In this way, given a continuous function $F:\mathcal{M} \to \mathbb{R}$
\begin{equation} \label{hel1} \int F d \hat{\mu} = \int_\mathcal{M} F(\delta_y) d \hat{\mu} (\delta_y) = \int_\Omega F(\delta_y) d \mu(y).
\end{equation}

\end{exam}

\begin{remark} \label{sim1}
Given $n$ denote by $\Gamma_n$ the equality distributed probability on  $\mathfrak{M}$ with support on the set
$$\Lambda_n =\{ \delta_x \,|\, \sigma^n (x) =x\}.$$

That is $ \Gamma_n =\frac{1}{m^n } \sum_{x \in \Lambda_n} \delta_{\delta_x},$ because $\# \Lambda_n = m^n.$

\smallskip
By compactness there exist a probability $\Pi^p$ on  $\mathfrak{M}$ such that for a convergent subsequence $\Gamma_{n_k} \to \Pi^p$, when $k \to \infty$. We call $\Pi^p$ the {\bf periodic preference} probability. As $\Gamma_n$ is $\mathfrak{T}$ invariant for each $n$, it follows  that   $\Pi^p $ is $\mathfrak{T}$-invariant.
\end{remark}





\smallskip


\smallskip

\section{Convolution and a contractive dynamics in the space of probabilities} \label{concon}

Given a continuous function $R: \Omega \times \Omega \to \Omega$ we will  define a product convolution $*:\mathcal{M} \times \mathcal{M}$. Take two probabilities $\nu,\mu \in \mathcal{M}$ we set
$$(\nu * \mu) (A) =[ \nu \times \mu]\ (R^{-1}(A))$$
in the sense that for any continuous function $f:\Omega \to \mathbb{R}$
$$\int_\Omega f d  (\nu * \mu) =   \int_\Omega f  (R (x,y) )\, d \nu (x) \, d \mu(y).$$

$\nu * \mu$ is a new probability in $\mathcal{M}.$



\smallskip

We refer the reader to  \cite{Uggioni} and \cite{Bar} for results considering distinct  concepts of convolution that are  different from ours.

\begin{lemma} \label{con1} Given a convolution $*$ obtained from $R$, for a fixed $\eta$,  the map  $\mu \to \eta * \mu$ é
is $s$-Lipschitz with respect to Monge-Kantorovich distance, provided that  $R$ is  $s$-Lipschitz w.r.t. the second variable.
\end{lemma}

\begin{proof}
Indeed, consider $\mu, \mu' \in \mathcal{M}$ then
$$d_{MK}(\eta * \mu, \eta * \mu')= \sup_{Lip(f) \leq 1} \int f d (\eta * \mu) -\int f d (\eta * \mu')= $$
\begin{equation} \label{GG} \sup_{Lip(f) \leq 1}\int f  (R (x,y) )\, d \eta(x) \, d \mu(y) - \int f  (R (x,y) )\, d \eta (x) \, d \mu'(y).\end{equation}
Defining $g(y):=\int f  (R (x,y) )\, d \eta(x), \; \forall y \in \Omega$ we get,
$$|g(y) - g(y')| =| \int f  (R (x,y) )\, d \eta(x) - \int f  (R (x,y') )\, d \eta(x)| \leq $$ $$\leq \int  Lip(f) \cdot d_{\Omega}(R (x,y), R (x,y')) d \eta(x) \leq s d_{\Omega}(y, y')$$
thus $Lip(\frac{1}{s}g ) \leq 1$.  Returning to expression \eqref{GG} we obtain
$$d_{MK}(\eta * \mu, \eta * \mu') \leq s \sup_{Lip(f) \leq 1} [\int \frac{1}{s}g(y) \,d \eta (x) \, d \mu(y) - \int \frac{1}{s}g(y)\, d \eta (x) \, d \mu'(y)] \leq $$ $$\leq s \cdot d_{MK}(\mu, \mu').$$
\end{proof}

\begin{corollary} \label{este1}
  Let $\eta_{j}, \; j=1,2,...$ be a sequence of probabilities on $\Omega$ and $R: \Omega^2 \to \Omega$ a convolution kernel which is $s$-Lipschitz contractive w.r.t. second variable. Then the CIFS(countable iterated function system) $\mathfrak{R}=(\Omega, \phi_{j})$, $j \in \mathbb{N}$,  where $\phi_{j}(\mu)=\eta_{j}*\mu$, is uniformly contractible with Lipschitz constant $s$.

\end{corollary}

\begin{lemma} If $R(x,y)=R(y,x)$ we get for the associated convolution $*$:
$$\mu * \nu = \nu * \mu.$$
\end{lemma}

\begin{proof} $\forall f \in C(\Omega)$
$$\int_\Omega f d (\nu *\mu) =\int_\Omega f (R(x,y)) d \nu (x) d \mu(y)=$$
$$\int_\Omega f (R(y,x)) d \nu (x) d \mu(y)=\int_\Omega f d (\mu *\nu).$$
\end{proof}

 The next example will exhibit the concept of convolution that we will use here (which is not commutative).

\begin{exam} \label{conc}
For example, given $n \in \mathbb{N}$, we can get  a {\bf product convolution} $*_n$ in $\mathcal{M}$ via  $R_n (x,y)=R_n(x,y)=(\pi_n(x) ,y_1,y_2,...) =(x_1,..., x_n, y_1, y_2, ...)$, where  $\pi_{n}(x)=(x_1,..., x_n)$. In this case
$$d_{\Omega}(R_n (x,y), R_n (x,y')) \leq \frac{1}{2^n} d_{\Omega}(y, y'),$$
and thus $R_n$ is $\frac{1}{2^n}$-Lipschitz w.r.t. $y$.

The ${*}_n$ convolution is defined for pairs  of probabilities $\eta,\mu$ in $\mathcal{M}$: we set
$\eta {*}_n \mu \in \mathcal{M} $ as the probability such that for any continuous function $f:\Omega \to \mathbb{R}$
$$\int f (z) d  (\eta *_n \mu) (z) =   \int\int  f  (R_n (x,y) )\, d \eta (x) \, d \mu(y)=$$
$$ \int\int f(x_1,..., x_n, y_1, y_2, ...)\,d \eta (x) \, d \mu(y).$$

This product convolution is not commutative.

 \end{exam}

 \begin{exam} \label{umum}
 For example, when $n=1$, we  write  $\mu \to \eta *_1 \mu$.  One can show that
 $$(\eta*_1\mu) * \mu= (\eta*_1\mu).$$
 \end{exam}
 
 We leave the proof to the reader. Note that $(\eta_1*_1\mu) *_1 \mu$ is different from $ \eta *_1 (\eta *_1 \mu) $.

 \begin{exam} \label{umum} Now we introduce the dynamics of $\mathfrak{T}$, and at the same time we will combine it with the convolution  $\mu \to \eta *_1 \mu$.
 In this way, for any continuous function $f:\Omega \to \mathbb{R}$
 $$ \int f(z)\,d (\mathfrak{T} (\nu) *_1\, \mu )(z)=\int \int f (R_1(\sigma(x) ,y)) d\nu (x) \, d \mu(y)= $$
 \begin{equation} \label{um}\int \int f(\pi_1(\sigma(x)),y)d \nu (x) \, d \mu(y)=  \int \int f(x_2,y)d \nu (x) \, d \mu(y).  
 \end{equation}

 If $\nu$ is $\sigma$-invariant then 
  \begin{equation} \label{umde} \int f(z)\,d (\mathfrak{T} (\nu) *_1\mu )(z)=  \int \int f(x_1,y)d \nu (x) \, d \mu(y).  
 \end{equation}

  In this case is not necessarily true that $\mathfrak{T} (\nu) *_1\mu$ is $\sigma$-invariant, even if $\mu$ is  $\sigma$-invariant.

 Moreover,
 \begin{equation} \label{um2} 
  \mathfrak{T}(\delta_x) *_1\delta_z = \delta_{x_2, z},
 \end{equation} 
 where $x=(x_1,x_2,...,x_n,..)\in \Omega$ and $z \in \Omega$.

 Note that 
 \begin{equation} \label{um13} 
  \int f(z) d (\mathfrak{T}(\delta_x) *_1\mu) (z) = \int f(x_2,y) \, d \mu(y).
 \end{equation}

  If $\sigma(x)=x$, then 
\begin{equation} \label{um3} 
  \mathfrak{T}(\delta_x) *_1\delta_z = \delta_{x_1, z}.
 \end{equation}

 

 If we denote $\psi_\nu (\mu)= \mathfrak{T}(\nu)*_1 \mu$,  then 
  \begin{equation} \label{ilu234 }\psi_{\nu_2} (\psi_{\nu_1} (\mu))=\mathfrak{T}(\nu_2) *_1\, ( \mathfrak{T}(\nu_1)*_1 \mu)
  \end{equation}  is such  that for a continuous function $A:\Omega \to \mathbb{R}$
 $$  \int A(z) d  \psi_{\nu_2} (\psi_{\nu_1} (\mu))(z)=\int A(z)d\,  [\, \mathfrak{T}(\nu_2) *_1 (\mathfrak{T}(\nu_1)*_1 \mu  )\,] (z) $$
 $$\int \int A (\pi_1(\sigma(x)),y)    \,d  \nu_2 (x)\,d (\mathfrak{T}(\nu_1)*_1 \mu) (y) \, =$$
 $$\int \int  \int A( \pi_1 (\sigma(x)),\,\pi_1 (\sigma(u))\, \,,v)\, d \nu_1(u)\,\, d \mu(v) \,d \nu_2\,(x)= $$
 \begin{equation} \label{ilu1} \int \int \int \, A(x_2,u_2,v) d \nu_2(x) \, d \nu_1 (u) \, d \mu(v).\end{equation}\\
  If $\mu=\delta_{y}$, $\nu_2=\delta_{b}, \nu_1=\delta_{a}$, $a,b,y\in \Omega$, then 
 \begin{equation} \label{ilu2}   \int A\, d  \psi_{\nu_2} (\psi_{\nu_1} (\mu))=A( b_2 ,a_2  , y). 
 \end{equation}
 
\end{exam}

We present a particular example that will illustrate the theory.

\begin{exam} Given a probability $\nu$ and a probability $\mu$,
for any $n\in \mathbb{N}$
\begin{equation} \label{nimp} \mathfrak{T} (\delta_x) {*}_n \delta_y =\delta_{x_2,x_3,...,x_n,y},
\end{equation}
for $x=(x_1,x_2,...,x_n,..)$.

\smallskip
If $\nu$ is $\sigma$-invariant, then for any $n \in \mathbb{N}$
\begin{equation} \label{nimpn} \int f d (\mathfrak{T}(\nu) {*}_n \mu)=   \int  \int f(x_1,x_2 ,..., x_n,y) \, d \nu (x) \, d \mu(y).  
\end{equation}
\end{exam}

We leave the proof to the reader.

\section{IFSs in  probability spaces} \label{IFS1}
There are two main ways to introduce an IFS in $\mathcal{M}$, using a compact number of maps (which includes the case of a finite number) and using a noncompact (usually countable one) number of maps.\\
\subsection{The compact model and holonomic probabilities}  \label{oo1}

It was introduced in \cite{BOR23} the concept of IFS with measures (IFSm for short). In this case, $(X,d)$ is a compact metric space and $\Lambda$ is another compact space, $R=(\phi_{\lambda}, q:=(q_{x}))_{\lambda \in \Lambda}$ where $\phi_{\lambda}: X \to X$ are continuous maps and $q={(q_{x})}_{x\in X}$ is a collection of measures on $\Lambda$ for all $x \in X$, such that
\begin{enumerate}
  \item[H1] $\sup_{x\in X} q_{x}(\Lambda) < \infty$,
  \item[H2] $\inf_{x\in X} q_{x}(\Lambda) > 0$,
  \item[H3] $ x \mapsto  q_{x}(A)$ is a Borel map, i.e, is $\mathcal{B}(X)$-measurable for all fixed $A\in \mathcal{B}(\Lambda)$,
  \item[H4] $ x \mapsto  q_{x}$ is weak-$\ast$-continuous.
\end{enumerate}

The transfer operator acts on continuous functions $f:X \to \mathbb{R}$  is given by
\begin{equation} \label{esteB1}B_{q}(f)(x)= \int_{\Lambda} f(\phi_{\lambda}(x)) dq_{x}(\lambda).
\end{equation}

The dual operator acts in probabilities $\rho$ on $X$ via Riesz Theorem:
\begin{equation} \label{esteB2}B_{q}^*(\rho)(f)= \int_X B_{q}(f)(x) d\rho(x).
\end{equation}

Below we consider the convolution $*_n$, where $n$ is fixed,  previously defined   in Example \ref{conc}.

Our setup is:\\
1) $X=\mathcal{M}$, compact and $d=d_{MK}$;\\
2) $\Lambda=\mathcal{M}$, compact;\\
3) $\phi_{\nu}(\mu)=\mathfrak{T}(\nu)*_n\mu$;\\
4) $dq_{\mu}(\nu):=e^{A(\phi_{\nu}(\mu))} d \Pi_0(\nu)$, where $A$ is  a continuous potential $A:\mathcal{M} \to \mathbb{R}$ and $\Pi_0\in \mathfrak{M}$ is a fixed a priori probability over $\mathcal{M}$.

Thus, we will consider here the IFSm
\begin{equation} \label{esteS} S=(\mathcal{M},   \phi_{\nu},   q_{\mu})_{\nu \in \mathcal{M}}.
\end{equation}

Then, for a fixed $n \in \mathbb{N}$, we can write the transfer operator $ B_{\Pi_0}:=B_{\Pi_0,A,T} $ as: 
$$B_{\Pi_0}(F)(\mu)= \int_{\mathcal{M}} F(\mathfrak{T}(\nu)*_n\mu) dq_{\mu}(\nu)= $$
\begin{equation} \label{esteB3} \int_{\mathcal{M}} e^{A(\phi_{\nu}(\mu))} F(\mathfrak{T}(\nu)*_n\mu) d\Pi_0(\nu),
\end{equation}
for any continuous function $F:\mathcal{M} \to \mathbb{R}$

We will see that, under mild assumptions, the definitions 1), 2), 3), and 4), mentioned above in our setup satisfy the required hypothesis described in \cite{BOR23}, so we can derive the standard properties obtained in classical thermodynamic formalism for our  IFSm \eqref{esteS}.

Indeed, the above hypothesis (H1)-(H4) from \cite{BOR23} are trivially satisfied for $dq_{\mu}(\nu)=e^{A(\phi_{\nu}(\mu))} d \Pi_0 (\nu)$, if $A$ is at least continuous. But, some of the theorems will require more regularity from the IFS.

\smallskip

We say that $A: \mathcal{M} \to \mathbb{R}$ is $\Pi_0$-normalized if for any $\mu\in \mathcal{M} $ we get $B_{\Pi_0}(1)(\mu)=1$.
\medskip

\begin{exam}\label{prok} Following the above definition of  $ B_{\Pi_0}$ for $n=1$ consider a continuous function $\tilde{A}:\Omega \to \mathbb{R}$,  and for any $\rho \in \mathcal{M}$ we set $A(\rho)  =\, \int \tilde{A}\,  d \rho.$

Such potential $A$ satisfies the necessary conditions of the future  Theorem \ref{powers-description}.

Take $\Pi_0= \frac{1}{m} \sum_{j=1}^m  \delta_{ \delta_{(j,j,j,..,j..)}  }\in \mathfrak{M}.$ Considering the probability $\mu=\delta_y$, according to \eqref{um3} we get that 
\begin{equation} \label{esteBc7} \phi_{\delta_{(j,j,j,..,j..)}}(\mu)=\mathfrak{T}(\delta_{(j,j,j,..,j..)})*_1\mu=\mathfrak{T}(\delta_{(j,j,j,..,j..)})*_1\delta_y= \delta_{j,y} .\end{equation}

Therefore, for $\nu=  \frac{1}{m}\,\delta_{(j,j,j,..,j..)}$, we get from \eqref{um2} 
$$A ( \phi_\nu\,(\mu) ) =A ( \,\phi_{\,\delta_{(j,j,j,..,j..)}}(\mu) ) = \tilde{A} (j,y).$$

Given the continuous  function $f:\Omega \to \mathbb{R}$, consider the continuous function $F:\mathcal{M} \to \mathbb{R}$ such that $F(\rho) =\, \int_\Omega  f \,d \rho.$
Then, we get from \eqref{esteBc7}, \eqref{um2}  and  \eqref{belo214}
$$B_{\Pi_0}(F)(\delta_y)=  \int_{\mathcal{M}} e^{A(\phi_{\nu}(\delta_y))} F(\mathfrak{T}(\nu)*_1\delta_y) d\Pi_0(\nu)= $$
$$\int_{\mathcal{M}} e^{A(\phi_{\nu}(\delta_y))} F(\mathfrak{T}(\nu)*_1\, \delta_y)\,d\,  \frac{1}{m} \sum_{j=1}^m \delta_{   \delta_{(j,j,j,..,j..)}  }(\nu)= $$
$$\frac{1}{m} \sum_{j=1}^m \int_{\mathcal{M}} e^{A(\phi_{ \,\delta_{(j,j,j,..,j..)} }(\delta_y))} F(\mathfrak{T}( \delta_{(j,j,j,..,j..)} )*_1\delta_y)= $$
\begin{equation} \label{esteB37}\frac{1}{m}  \sum_{j=1}^m  e^{\, A(\delta_{j,y})  }    F( \delta_{j,y})  =\frac{1}{m}  \sum_{j=1}^m  e^{\tilde{A} (j,y)}    f(j,y)=   \mathcal{L}_{\tilde{A}} (f)(y).
\end{equation}
\end{exam}

\begin{remark} \label{prt}The last expression describes the action of the classical Ruelle operator for the a priori probability $\frac{1}{m} \sum_{j=1}^m \delta_j$  and the potential $\tilde{A}$ (see \cite{LMMS} or\cite{L1}).
Therefore, in some sense, the above definition of $B_{\Pi_0}$ is a Level-2 version of the classical Ruelle operator.
\end{remark}

\begin{exam} \label{orw}
Note that for $n=1$ we get $\phi_{\nu_2} (\phi_{\nu_1}(\mu))=\mathfrak{T}(\nu_2) *_1 ( \mathfrak{T} (\nu_1) {*}_1 \mu),  $ a case which was discussed in expression  \eqref{ilu1}.

In this case
$$B_{\Pi_0}^2 \, (F)(\delta_y)= $$
$$  \int_{\mathcal{M}} \int_{\mathcal{M}}e^{A(\phi_{\nu_2} (\phi_{\nu_1}(\delta_y)))+A (\phi_{\nu_1}(\delta_y))  } F(\,\phi_{\nu_2}\,(\mathfrak{T}(\nu_1)*_1\delta_y)\,) \,d\,\Pi_0(\nu_1)  d\,\Pi_0(\nu_2)= $$
$$ \sum_{r,s=1}^m  e^{\tilde{A} (r,s,y)+\tilde{A} (s,y) }    f(r,s,y).$$

In the general case we get that for any $\mu \in \mathcal{M}$
$$ B_{\Pi_0}^2(F)(\mu)=$$
\begin{equation} \label{esteB35}\int_{\mathcal{M}}   \int_{\mathcal{M}} e^{A(\phi_{\nu_2} (\phi_{\nu_1}(\mu)))+A (\phi_{\nu_1}(\mu))  } F(\phi_{\nu_2} (\phi_{\nu_1}(\mu)))) d\Pi_0(\nu_1)\,  d\Pi_0(\nu_2).
\end{equation}

\end{exam}

\smallskip

Given the potential $A$ we will derive a probability $\Pi_A\in \mathfrak{M}$ which will play the role of the Gibbs probability for the potential $A$ (see Definition \ref{Gler}).


A natural choice for the a priori probability $\Pi_0$ is the probability $\Pi^p$ which was described above.

We recall the main results derived from \cite{BOR23} (when applied to our setting):

\begin{theorem}\cite[Theorem 2.5]{BOR23}\label{powers-description}
Denote by  $S$ the  IFSm described by \eqref{esteS}
and suppose that there is a positive number $\lambda$ and
a strictly positive continuous function
$h: \mathcal{M} \to \mathbb{R}$ such that  $B_{\Pi_0}(h)=\lambda h$.
Then the following limit exists
\begin{align}\label{1overNlnBN}
\lim_{N \to \infty}
\frac{1}{N}
\ln\left(B_{\Pi_0}^{N}(1) (\mu) \right)
=
\log \lambda
\end{align}
the convergence is uniform in $\mu\in \mathcal{M}$ and
$\lambda=\lambda(B_{\Pi_0})$ is the spectral radius of $B_{\Pi_0}$ acting on $C(\Omega,\R)$.
\end{theorem}


In our case, the family of measures satisfies the requirements from \cite{BOR23}. Indeed, as  $d q_{\mu}(\nu)=e^{A(\phi_{\nu}(\mu))} d \Pi_0 (\nu)$,  we get that $u(\mu, \nu):= \log \frac{dq_{\mu}}{d \Pi_0}(\nu) = A(\phi_{\nu}(\mu))$ has the regularity prescribed in \cite{cioletti2019}.

\smallskip

Note that in the case $A$ is $\Pi_0$-normalized, that is $B_{\Pi_0}(1)=1$, we get that $\lambda=1$ and $h=1$.

\smallskip

\begin{theorem}\cite[Theorem 2.6]{BOR23}\label{exist posit eigenfunction}
  Let $S$ be the IFSm described by \eqref{esteS}. If $A$ is Lipschitz, then there exists a positive and continuous eigenfunction $h: \mathcal{M} \to \mathbb{R}$ such that $B_{\Pi_0}(h) = \lambda(B_{\Pi_0})h$.
\end{theorem}

\begin{definition} \label{DEFIG} Given  $A$ and $\Pi_0$ we say that $\hat{\Pi}=\hat{\Pi}_{A, \Pi_0}$ is eigenprobability for $A$  and $\Pi_0$ if
there exist a positive number  $\lambda$ such that for all continuous $F:\mathcal{M} \to \mathbb{R}$
$$B_{\Pi_0}^* ( \hat{\Pi} ) = \lambda \hat{\Pi}.$$

This means that for any $F:\mathcal{M} \to \mathbb{R}$ we get
$$ \lambda \,\int_{\mathcal{M}}  F(\rho) d \hat{\Pi} (\rho)= \int_{\mathcal{M}} B_{\Pi_0}(F) (\mu)\,d\,\hat{\Pi}(\delta_\mu)=$$
\begin{equation} \label{ou1}  \int_{\mathcal{M}} \,(  \int_{\mathcal{M}} e^{A(\phi_{\nu}(\mu))} F(\phi_{\nu}(\mu))) d\Pi_0(\nu)\,)  \, d\,\hat{\Pi}(\delta_\mu).
\end{equation}

\end{definition}

\begin{remark}  Note that the eigenvalue $\lambda$ is identified when we apply  \eqref{ou1}   to the function $F=1.$ Moreover,
 \eqref{ou1} shoud be true for functions of the form $F(\rho) = \int f d \rho$, where we take  a fixed continuous function $f:\Omega \to \mathbb{R}$.
\end{remark}

\begin{remark} From \eqref{abu1} the equation \eqref{ou1}  is equivalent to
$$  \lambda \,\int_{\mathcal{M}}  F(\rho) d \hat{\Pi} (\rho)=\lambda \, \int_\mathcal{M} F(\mu)\, d \mathfrak{O}_{\hat{\Pi}} (\delta_\mu) =  $$
$$\int_{\mathcal{M}} \int_\mathcal{M} \,  [\, \int_{\mathcal{M}} e^{A(\phi_{\nu}(\rho))} F(\phi_{\nu}(\rho)) d\Pi_0(\nu)\,)\,\,]\,d \tilde{\Pi}\, (\rho)\,d  \mathfrak{O}_{\hat{\Pi}} (\tilde{\Pi})= $$ 
$$\int_{\mathcal{M}} \int_\mathcal{M} \,  [\, \int_{\mathcal{M}} e^{A(\phi_{\nu}(\rho))} F(\phi_{\nu}(\rho)) d\Pi_0(\nu)\,)\,\,]\,d \delta_\mu\, (\rho)\,d  \mathfrak{O}_{\hat{\Pi}} (\delta_\mu) $$
\begin{equation} \label{abu2}\int_\mathcal{M} \,  [\, \int_{\mathcal{M}} e^{A(\phi_{\nu}(\mu))} F(\phi_{\nu}(\mu)) d\Pi_0(\nu)\,)\,\,]\,\,d  \mathfrak{O}_{\hat{\Pi}} (\delta_\mu).    
\end{equation}

We will present several examples  always taking $n=1$ in our main theorem above.
 

\end{remark}

\begin{exam}  \label{prof} Assume the hypothesis of Example \ref{prok}. Here we will apply the reasoning of Remark \ref{kui}.

Take  $\Pi_0= \frac{1}{m} \sum_{j=1}^m\delta_{   \delta_{(j,j,j,..,j..)}  }$ and a continuous potential $A:\mathcal{M} \to \mathbb{R}$. 

We will show that we can describe  the eigenprobability $\hat{\Pi}$ for $\Pi_0$ and $A$ of Definition \ref{DEFIG} via the
eigenprobability $\mu_{ B}$ for the dual of the Ruelle  operator $\mathcal{L}_B$ of a certain  continuous potential $B:\Omega \to \mathbb{R}$. Consider a continuous function $F:\mathcal{M} \to \mathbb{R}.$

We assume $\hat{\Pi}$ satisfies equation \eqref{abu2} for some $\lambda.$ From \eqref{abu2} and \eqref{um3} this means
$$ \lambda \,\int_{\mathcal{M}}  F(\rho) d \hat{\Pi} (\rho)= \lambda \, \int_\mathcal{M}  F (\mu ) d \mathfrak{O}_{\hat{\Pi}} (\delta_\mu)  = $$
$$ \int_\mathcal{M} \,  [\, \int_{\mathcal{M}} e^{A(\phi_{\nu}(\mu))} F(\phi_{\nu}(\mu)) d\Pi_0(\nu)\,)\,\,]\,\,d  \mathfrak{O}_{\hat{\Pi}} (\delta_\mu)=$$
\begin{equation} \label{ful2}  \int_{\mathcal{M}} \frac{1}{m}  \sum_{j=1}^m  e^{A( \phi_{\delta_j^\infty}(\mu) )} F(\phi_{\delta_j^\infty}(\mu))   \,d  \mathfrak{O}_{\hat{\Pi}} (\delta_\mu). 
\end{equation}

 


We will test if a probability $\hat{\Pi}$ that has support on probabilities of the form $\delta_{\delta_y}$, $y \in \Omega$, can satisfy \eqref{ful2}.  Using \eqref{pum000} and \eqref{um2}, 
in the affirmative case, this would imply that 
$$ \lambda \, \int_\Omega  F (\delta_y) d \mathfrak{O}_{\hat{\Pi}} (\delta_{\delta_y})  = $$  
$$\int_{\Omega} \frac{1}{m}  \sum_{j=1}^m  e^{A( \phi_{\delta_j^\infty}(\delta_y) )} F(\phi_{\delta_j^\infty}(\delta_y))   \,d  \mathfrak{O}_{\hat{\Pi}} (\delta_{\delta_y}) $$
\begin{equation} \label{ful4}  \int_{\Omega} \frac{1}{m}  \sum_{j=1}^m  e^{A( \delta_{(j,y)} )} F(\delta_{(j,y)})  \,d  \mathfrak{O}_{\hat{\Pi}} (\delta_{\delta_y}). 
\end{equation}

Next we will describe some expressions that will be useful in the future  Example \ref{abu1}.

Consider the continuous  potential $B(r) =B(r_1,r_2,..)=  A( \delta_{r})$ and the equilibrium probability
$\mu_B$ associated to the corresponding eigenvalue $\beta$.
Denote $G(y)= F( \delta_y)$.

Then, we get
$$ \beta \,\int_\Omega G(y) d \mu_B( y)= $$
\begin{equation} \label{ful5}  
\int_\Omega \mathcal{L}_A (G) (y) d  \mu_B( y)= \int_{\Omega} \frac{1}{m}  \sum_{j=1}^m  e^{A( \delta_{(j,y)} )} G(j,y))   d  \mu_B( y).
\end{equation}

Therefore, taking $d  \mathfrak{O}_{\hat{\Pi}} (\delta_{\delta_y})$ as $d \mu_B(y)$, and $\lambda=\beta$ we get that equality  \eqref{ful5} is equivalent to
equality \ref{ful4}.

The final conclusion is that $d\,\mathfrak{O}_{\hat{\Pi} }(\delta_{\delta_y})$ can be taken as $d \mu_B (y)$.

\smallskip

For such class of  potentials $A$ and such a priori  $\Pi_0$, the action of $\hat{\Pi}$ in each continuous function  $F$ is given by \eqref{abu1}

$$  \int_\mathcal{M} F (\beta) d \hat{\Pi}(\beta) =  \int_\Omega \,  [\,\int_{  \mathfrak{M}}\,F(\gamma) \, d \delta_{\delta_y} (\gamma) \,\,]\, d\,\mathfrak{O}_{\hat{\Pi}} (\delta_{\delta_y})= $$
\begin{equation} \label{abu91}  \int_\Omega \,  [\,\int_{  \mathcal{M}}\,F(\gamma) \, d \delta_{\delta_y} (\gamma) \,\,]\, d\,\mu_{B} (y)= \int_\Omega F(\delta_y) \, d\,\mu_{B} (y).
\end{equation}



 
\end{exam}

From now on, we can assume that the operator  is $\Pi_0$-normalized, that is $B_{\Pi_0}(1) = 1$, otherwise, we can replace the measures by
  \[ p_{\mu}(\nu) = \frac{h(\phi_{\nu}(\mu))}{\rho h(\mu)} q_{\mu}(\nu)\]
obtaining $B_{\Pi_0}(1) = 1$. Note, however, that in this procedure  we may lose the knowledge about the regularity of  $p_{\mu}$.

\begin{definition} \label{for1} If $A$ is $\Pi_0$-normalized we say that $\hat{\Pi}=\hat{\Pi}_{A, \Pi_0}$ is Gibbs if
$$B_{\Pi_0}^* ( \hat{\Pi} ) = \hat{\Pi}.$$

This means that for any $F:\mathcal{M} \to \mathbb{R}$ we get
$$ \int_{\mathcal{M}}  F(\rho) d \hat{\Pi} (\rho)=  \int_{\mathcal{M}} B_{\Pi_0}(F) (\mu)\,d\,\hat{\Pi}(\delta_\mu)=$$
\begin{equation} \label{esteBB38}  \int_{\mathcal{M}} \,(  \int_{\mathcal{M}} e^{A(\phi_{\nu}(\mu))} F(\phi_{\nu}(\mu)\,) d\Pi_0(\nu)\,)  \, d\,\hat{\Pi}(\delta_\mu).
\end{equation}

\end{definition}

\begin{exam}  \label{prof1} Assume the hypothesis of Example \ref{prof}.

Take   $\Pi_0= \frac{1}{m} \sum_{j=1}^m\delta_{   \delta_{(j,j,j,..,j..)}  }$ and assume that $A$ is normalized. We will show the existence of $\mathfrak{T}$-invariant probabilities.

We  showed in Example \ref{prof} that we can describe  the eigenprobability $\hat{\Pi}$ of Definition \ref{DEFIG} (or the one in Definition \ref{for1}) via the
eigenprobability $\mu_{ B}$ for the Ruelle  operator of a certain continuous potential $B$. We take $G(y) = F(\delta_y).$

Therefore,  can recover for such class of  potentials $A$, the action of $\hat{\Pi}$ in each continuous function  $F$ via

 $$ \int_\mathcal{M} F (\rho) d \hat{\Pi}(\rho) =  \int_{\mathcal{M}} \,\,\,(\int F\, d\delta_{\delta_y}) \, d \mu_{B}(y)=\int G(y)  d \mu_{B}(y).     $$

Therefore, from the above and \eqref{pum000} 

 $$ \int_\mathcal{M}( F \circ \mathfrak{T})  (\rho) d \hat{\Pi}(\rho) =  \int_{\mathcal{M}} \,\,\,\left(\int( F\circ \mathfrak{T})\, d\delta_{\delta_y}\right) \, d \mu_{B }(y)=$$
$$ \int G( \sigma(y))  d \mu_{B }(y) =   \int G(y)  d \mu_{B }(y)= \int_\mathcal{M} F (\rho) d \hat{\Pi}(\rho),$$
because $\mu_{B }$ is $\sigma$-invariant. Then $\hat{\Pi}$ is $\mathfrak{T}$-invariant.

\end{exam}

\begin{exam}  \label{prok1}
Assume the hypothesis of Example \ref{prok}.  We will show the existence of normalized potentials $A:\mathcal{M} \to \mathbb{R}.$

Then, for  $A(\rho)  \, =\, \int_\Omega \tilde{A}\,  d \rho$, $\Pi_0= \frac{1}{m} \sum_{j=1}^m\delta_{   \delta_{(j,j,j,..,j..)}  }$, and  a special class of functions $F(\rho) \,=\, \int_\Omega  f (x)\,d \rho(x),$ we showed that for any $y$
$$B_{\Pi_0}(F)(\delta_y) =\mathcal{L}_{\tilde{A}} (f)(y). $$

We will  assume from now on that $\mathcal{L}_{\tilde{A}}(1)=1$.

Then, if $p=\sum_k p_k \delta_{y_k}\in \mathcal{M}$, where $\sum_k p_k=1$,  we get from  above that
$B_{\Pi_0}(1)(p)=1.$ As  $\tilde{A}$ (and also $A$)  is continuous and any probability in $\mathcal{M}$ can be approximated by probabilities of the form $p$, we get that $A$ is $\Pi_0$-normalized. This means
$$ 1=   \,\, \int_{\mathcal{M}} e^{A(\phi_{\nu}(\mu))} \,\, d\Pi_0(\nu).$$

Consider the Gibbs  probability $\hat{\Pi}=\hat{\Pi}_{A, \Pi_0}$ associate to the $\Pi_0$-normalized potential $A$. 
We will show a natural relation of  $ \hat{\Pi}  $ with $m_{\hat{\Pi} } $.

$\hat{\Pi}$  should satisfy in this case the property: for any $G:\mathcal{M} \to \mathbb{R}$ (not just for $F$ of the above form)
$$ \int_{\mathcal{M}}  G(\rho) d \hat{\Pi} (\rho)=\int_{\mathcal{M}} \,( \, \int_{\mathcal{M}} e^{A(\phi_{\nu}(\mu))} G(\phi_{\nu}(\mu))\,\, d\Pi_0(\nu)\,)  \, d\,\hat{\Pi}(\mu). $$

This should be true in particular for the case when $G$ is in the particular  form of the $F$ above. The above means for our choice of $\Pi_0$
$$  \int_{\mathcal{M}} \,\, \int_{\mathcal{M}} e^{A(\phi_{\nu}(\mu))} F( \phi_{\nu}(\mu) ) \, d\Pi_0(\nu)\, d\,\hat{\Pi}(\mu) =$$
$$ \int_{\mathcal{M}}   \frac{1}{m} \sum_{j=1}^m  e^{\, A(\phi_{ \delta_{j^\infty }}(\mu))}F(  \phi_{\delta_{j^\infty }} (\mu) ) d\,\hat{\Pi}(\mu) =$$
$$ \int_{\mathcal{M}}   \frac{1}{m} \sum_{j=1}^m  e^{\,\int  \tilde{A} (j,z)\, d \mu (z)\, }\,(\, \int  f (j,z)\, d \mu (z)\,)d\,\hat{\Pi}(\mu)= $$
$$\int_{\mathcal{M}}  F(\rho) d \hat{\Pi} (\rho) = \int ( \int f(x) d \rho(x) \,)\,d \hat{\Pi} (\rho) = \int f(y) d m_{\hat{\Pi} } (y)  .$$




\smallskip


\end{exam}

\begin{theorem} There exists a duality of the a priori $\Pi_0$ and  the eigenprobability $\hat{\Pi}$. Moreover, if we interchange them,  the eigenvalue  in \eqref{ou1} is the same, and furthermore
$$ m_{\Pi_0} = m_{\hat{\Pi}}.$$
\end{theorem}

\begin{proof} Given $A:\mathcal{M} \to \mathbb{R}$, we will show a relation of the a priori $\Pi_0$ and the eigenprobability $\hat{\Pi} $ for  $A$.

In this case we get for any $F:\mathcal{M} \to \mathbb{R}$
\begin{equation}  \label {ert1112}  \lambda \,\int_{\mathcal{M}}  F(\rho) d \hat{\Pi} (\rho)=\int_{\mathcal{M}} \,(\, \int_{\mathcal{M}} e^{A(\phi_{\nu}(\mu))} F( \phi_{\nu}(\mu) ) \, d\Pi_0(\nu)\,)\, d\,\hat{\Pi}(\mu) .
\end{equation}

Now suppose that for the potential $A$ we take  the a priori as $\hat{\Pi}$, and then  we get  the eigenprobability, denoted by $\Pi_1$, for this pair associated to to some eigenvalue $\beta>0.$ Then, for any $F$
\begin{equation}  \label {ert1122}  \beta \,\int_{\mathcal{M}}  F(\rho) d \Pi_1 (\rho)=\int_{\mathcal{M}} \,(\, \int_{\mathcal{M}} e^{A(\phi_{\nu}(\mu))} F( \phi_{\nu}(\mu) ) \,d\,\hat{\Pi}(\mu)\,) d\Pi_1(\nu)\,.
\end{equation}

If in the above equation, we set $\Pi_1 =  \Pi_0$ we get in \eqref{ert1122} the same expression as in
\eqref{ert1112}, up to the values $\lambda$ and $\beta$. From Remark \ref{kui} we get that $\lambda=\beta$. As $F$ is any continuous function we get that $\Pi_0$ is the eigenprobability for the a priori $\hat{\Pi}.$

\eqref{ou1} should be true in particular for the case when $F$ is in the particular  form
 \begin{equation}  \label {ert0}F(\rho)= \int f(x) d \rho(x),
\end{equation}
 for some fixed $f:\Omega \to \mathbb{R}$. This means for our choice of $\Pi_0$ and the eigenprobability $\hat{\Pi}$ that for any $f$
$$\lambda \int f(y) d m_{\hat{\Pi}} (y)= \lambda \,\int \, ( \int f(x) d \rho(x)) d \hat{\Pi}(\rho)=\lambda \int F(\rho ) d \hat{\Pi}(\rho) = $$
$$\int_{\mathcal{M}} \,(  \int_{\mathcal{M}} e^{A(\phi_{\nu}(\mu))} \, (\,\int f(x)  d\, \phi_{\nu}(\mu)(x)\,)\, d\Pi_0(\nu)\,)  \, d\,\hat{\Pi}(\mu)=  $$
\begin{equation}  \label {ert1} \int_{\mathcal{M}} \,(\, \int_{\mathcal{M}} e^{A(\phi_{\nu}(\mu))} F( \phi_{\nu}(\mu) ) \, d\Pi_0(\nu)\,)\, d\,\hat{\Pi}(\mu) .
\end{equation}

Note that $\lambda=\beta$ in \eqref{ert1112} and \eqref{ert1122}.


Now suppose that for $A$ we take the a priori $\hat{\Pi}$, and then we get that
$\Pi_1=\Pi_0$ is the eigenprobability for this pair and the eigenvalue $\lambda>0.$ For $F$ of the above form \eqref{ert0} we get from \eqref{ert1} 
$$\lambda \int f(y) d m_{\Pi_0} (y)= \lambda \, \int  (\int f(x) d \rho(x)) d \Pi_0 (\rho)=\lambda \int F(\rho ) d \Pi_0= $$
$$\int_{\mathcal{M}} \,(  \int_{\mathcal{M}} e^{A(\phi_{\nu}(\mu))} \, (\,\int f(x)  d\, \phi_{\nu}(\mu)(x)\,)\,d\,\hat{\Pi}(\mu) \,)   )d \Pi_0 (\nu) \,\,=  $$
\begin{equation}  \label {ert2}\int_{\mathcal{M}} \,(\, \int_{\mathcal{M}} e^{A(\phi_{\nu}(\mu))} F( \phi_{\nu}(\mu) ) \, \,\,d \Pi_0 (\nu))   d\,\hat{\Pi}(\mu)  \,= \lambda \int f(y) d m_{\hat{\Pi}} (y).
\end{equation}

As the equality is for any $f:\Omega \to \mathbb{R}$ we get that $m_{\hat{\Pi}}= m_{\Pi_0}.$

\end{proof}

\begin{theorem}\cite[Theorem 3.2]{BOR23}\label{prop-exist-auto-medida}
Let $S$ be the IFSm described by \eqref{esteS}.
Then there exists a positive number $\rho \leq \rho(B_{\Pi_0})$, such that the set
\[
\mathcal{G}^{*}(\Pi_0)
=
\{
  \Pi \in \mathfrak{M}: B_{\Pi_0}^*(\Pi) =\rho\,\Pi\}
\]
is not empty.
\end{theorem}

\begin{definition}\label{in}
Given the cartesian product space $\hat{\mathcal{M}}\equiv \mathcal{M} \times\Lambda= \mathcal{M} \times \mathcal{M}$, for each $f\in C(\mathcal{M},\mathbb{R})$ consider the ``$\Lambda$-differential''
$df: \hat{\mathcal{M}} \to \mathbb{R}$ which is  defined by $df[\mu](\nu)\equiv f(\phi_{\nu}(\mu)) -f(\mu)$.
\end{definition}

\begin{definition}\label{invariance}
A measure $\hat{\Pi}$ over $\hat{\mathcal{M}}$ is said holonomic,
with respect to the IFS $S$, if
for all $f\in C(\mathcal{M},\mathbb{R})$ we have
\begin{align*}
\int_{\hat{\mathcal{M}}}df[\mu](\nu) \, d\hat{\Pi}(\mu,\nu)=0.
\end{align*}
Notation,
\[
\mathcal{H}(S)\equiv
    \{\hat{\Pi}  \, | \, \hat{\Pi}
    \text{ is a holonomic probability measure with respect to the IFSm } S\}.
\]
\end{definition}

We now define the Variational Entropy of a holonomic measure.
\begin{definition}\cite[Definition 5.1, Theorem 5.6]{BOR23} or \cite{LO} for a preceding point of view. \label{entropy}
Let $S$ the IFSm described by \eqref{esteS}, $\hat{\Pi} \in \mathcal{H}(S)$, $Q$ any probability with support on $\mathcal{M}$, and
$d\hat{\Pi}(\mu,\nu)=d\Pi_{\mu}(\nu)d\pi(\mu)$ a disintegration of $\hat{\Pi}$.
The variational entropy of $\hat{\Pi}$ with respect to the a priori probability $Q$ is defined by
\[
  h_{v}^{Q}(\hat{\Pi})
   \equiv \inf_{ \substack{g\,\in\,C(\mathcal{M}, \mathbb{R}) \\ g>0 }}
        \left\{ \int_{\mathcal{M}} \ln\frac{B_{Q}(g)(\mu)}{ g(\mu) }  \,\text{d}\pi(\mu) \right\} \leq 0,
\]
where $B_{Q}(g)(\mu)=\int_{\mathcal{M}} g(\phi_{\nu}(\mu)) dQ(\nu)$.
\end{definition}

We will consider from now on the operator $B_{\Pi_0}$ as in \eqref{esteB3} and the variational entropy $h_{v}^{\Pi_0}$, where $\Pi_0$ is the fixed a priori probability on $\mathcal{M}.$

Recall that, for the IFSm S, $dq_{\mu}(\nu):=e^{A(\phi_{\nu}(\mu))} d \Pi_0 (\nu)$, for a  continuous potential  $A:\mathcal{M} \to \mathbb{R}$.

\begin{definition} Following \cite[Definition 5.8]{BOR23},  we define the topological pressure for the potential $\psi:=e^A: \mathcal{M} \to \mathbb{R}$ by
\[\mathbb{P}(\psi)
   \equiv \sup_{\hat{\Pi} \in \mathcal{H}(S)} \inf_{ \substack{g\,\in\,C(\mathcal{M}, \mathbb{R}) \\ g>0 }}
        \left\{ \int_{\mathcal{M}} \ln\frac{B_{\Pi_0}(g)(\mu)}{ g(\mu) }  \,\text{d}\pi(\mu) \right\} \leq 0,
\]
where
\begin{equation} \label{kler}d\hat{\Pi}(\mu,\nu)=d\Pi_{\mu}(\nu)d\pi(\mu)
\end{equation}
is the disintegration of $\hat{\Pi}$, with respect to  the marginal $\pi$.
\end{definition}

\begin{proposition}\cite[Lema 5.9]{BOR23} The pressure satisfies 
  $$\mathbb{P}(\psi)=\sup_{\hat{\Pi} \in \mathcal{H}(S)} h_{v}^{\Pi_0}(\hat{\Pi}) + \int_{\mathcal{M}} \ln( \psi(\mu))  \, \text{d}\pi(\mu)$$
  \begin{equation} \label{kler39}\sup_{\hat{\Pi} \in \mathcal{H}(S)} h_{v}^{\Pi_0}(\Pi) + \int_{\mathcal{M}} A(\mu)  \, \text{d}\pi(\mu)
\end{equation}
\end{proposition}

\begin{definition} \label{lkle} A holonomic probability $\hat{\Pi}_A \in \mathcal{H}(S)$ satisfying the equality
\[\mathbb{P}(\psi)= h_{v}^{\Pi_0}(\hat{\Pi}_A) + \int_{\mathcal{M}} A(\mu)  \, \text{d}\pi_A(\mu),
\]
where $\pi_A$ comes from the disintegration of $\hat{\Pi}_A$ (as in \eqref{kler}), is called an equilibrium state for the potential $A:\mathcal{M} \to \mathbb{R}$.
\end{definition}

From \cite[Theorem 5.13]{BOR23} the set of equilibrium states is not empty for the IFSm S, since $dq_{\mu}(\nu):=e^{A(\phi_{\nu}(\mu))} dP(\nu)$ (a continuous and positive weight).
\begin{remark}
As we already show that there exists a positive eigenfunction for $B_{\Pi_0}$ (Theorem~\ref{exist posit eigenfunction}) and an eigenmeasure for  $B_{\Pi_0}^*$ (Theorem~\ref{prop-exist-auto-medida}), then it follows from \cite{BOR23}  that the pressure obtained by the entropy with respect to the a priori measure $\Pi_0$ satisfy $\mathbb{P}(\psi)= \ln(\rho(B_{\Pi_0}))$. Thus an equilibrium measure  $\hat{\Pi}_A$ satisfies
\[\ln(\rho(B_{\Pi_0})) = h_{v}^{\Pi_0}(\hat{\Pi}_A) + \int_{\mathcal{M}} A(\mu)  \, \text{d}\pi(\mu_A).
\]
\end{remark}

Recall that the projection $\Phi$ from $\hat{\Pi}$ over $\hat{\mathcal{M}}$ to $\mathfrak{M}$, defining  $\Pi = \Phi(\hat{\Pi})$, is given by
$$\int_{\mathcal{M}} g(\mu) d \Phi(\mu)=\int_{\mathcal{M}} g(\mu) d\Phi(\hat{\Pi})(\mu):= \int_{\hat{\mathcal{M}}} g(\mu) d\hat{\Pi}(\mu, \nu), \forall g.$$

\begin{definition} \label{Gler}
The probability $\Pi_A= \Phi(\hat{\Pi}_A)\in \mathfrak{M}$ is caled the projected equilibrium probability for $A$ and the a priori probability $\Pi_0\in \mathfrak{M}$.
\end{definition}

Consider the functional $m:C(\mathcal{M},\mathbb{R})\to \mathbb{R}$
given by
\begin{align}\label{def-funcional-p}
m(A) = \mathbb{P}(e^A).
\end{align}
It is immediate to verify
that $m$ is a convex and a finite valued functional.

\begin{theorem}\cite[Theorem 6.1, Corollary 6.2]{BOR23}
   Consider the IFSm $S$. If $m$ is G\^ateaux differentiable in $A$ then
   \[
\# \{\Phi(\hat{\Pi}):  \hat{\Pi}  \text{ is an equilibrium state for }  \psi=e^A \} =1.
\]
\end{theorem}

\end{document}